\documentclass[12pt]{amsart}

\usepackage{amssymb}
\usepackage{dsfont}

\usepackage{enumerate}

\usepackage{graphicx}

\makeatletter
\@namedef{subjclassname@2010}{%
  \textup{2010} Mathematics Subject Classification}
\makeatother

\usepackage[T1]{fontenc}

\newtheorem{thm}[equation]{Theorem}
\newtheorem{cor}[equation]{Corollary}
\newtheorem{prop}[equation]{Proposition}

\theoremstyle{definition}

\numberwithin{equation}{section}

\begin{document}


\baselineskip=17pt



\title[Leibniz's rule on two-step nilpotent Lie groups]{Leibniz's rule \\ on two-step nilpotent Lie groups}
\author[K. Beka{\l}a]{Krystian Beka{\l}a}
\address{Institute of Mathematics
\\ University of Wroc{\l}aw
\\ 50-384 Wroc{\l}aw, Poland}
\email{krystian.bekala@math.uni.wroc.pl}

\date{}

\begin{abstract}
Let $\mathfrak{g}$ be a nilpotent Lie algebra which is also regarded as a homogeneous Lie group with the Campbell-Hausdorff multiplication. This allows to define a generalized multiplication $f \# g = (f^{\vee} * g^{\vee})^{\wedge}$ of two functions in the Schwartz class $\mathcal{S}(\mathfrak{g}^{*})$, where $\vee$ and $\wedge$ are the Abelian Fourier transforms on the Lie algebra $\mathfrak{g}$  and on the dual $\mathfrak{g}^{*}$.

In the operator analysis on nilpotent Lie groups an important notion is the one of symbolic calculus which can be viewed as a higher order generalization of the Weyl calculus for pseudodifferential operators of H{\"o}rmander. The idea of such a calculus consists in describing the product $f \# g$ for some classes of symbols.

We find a formula for $D^{\alpha}(f \# g)$ for Schwartz functions $f,g$ in the case of two-step nilpotent Lie groups, that includes the Heisenberg group. We extend this formula to the class of functions $f,g$ such that $f^{\vee}, g^{\vee}$ are certain distributions acting by convolution on the Lie group, that includes usual classes of symbols.
In the case of the Abelian group $\mathds{R}^{d}$ we have $f \# g = fg$, so $D^{\alpha}(f \# g)$ is given by the Leibniz rule.
\end{abstract}

\subjclass[2010]{Primary 22E25; Secondary 22E15}

\keywords{Leibniz's rule, Heisenberg group, Fourier transform, homogeneous groups, symbolic calculus, convolution}

\maketitle

\section{Statement of the result} \label{secone}
Let $\mathfrak{g}$ be a nilpotent Lie algebra of the dimension $d$ which is endowed with a family of dilations $(\delta_{t})_{t>0}$.
We also regard the vector space $\mathfrak{g}$ as a Lie group with the multiplication law given by the Campbell-Hausdorff formula (see Corwin-Greenleef \cite{CoGr})
\begin{equation*}
x \circ y = x + y + r(x,y),
\end{equation*}
where $r(x,y)$ is the (finite) sum of the commutator terms of order at least 2 in the Campbell-Hausdorff series for $\mathfrak{g}$. 

This allows to define a generalized multiplication $f \# g = (f^{\vee} * g^{\vee})^{\wedge}$ of two functions in the Schwartz class $\mathcal{S}(\mathfrak{g}^{*})$, where $\vee$ and $\wedge$ are the Abelian Fourier transforms on the Lie algebra $\mathfrak{g}$  and on the dual $\mathfrak{g}^{*}$.
In the case of the Abelian group $\mathds{R}^{d}$, one gets $f \# g = fg$.

In the operator analysis on nilpotent Lie groups an important notion is the one of symbolic calculus which can be viewed as a higher order generalization of the Weyl calculus for pseudodifferential operators of H{\"o}rmander \cite{Hor}. The calculus was created in Melin \cite{Mel} and developed in Manchon \cite{Man} and G{\l}owacki \cite{Glo4}, \cite{Glo3}, \cite{Glo1}. The idea of such a calculus consists in describing the product $f \# g$ for some classes of symbols. One of the obstacles in extending Weyl calculus to the ground of general nilpotent Lie groups is the lack of formula allowing to calculate the derivatives of the product $f \# g$.

In the Abelian case, we have the multidimensional Leibniz rule
\begin{equation} \label{leib}
D^{\alpha} (fg) =\sum_{\beta+\gamma=\alpha} {\alpha \choose \beta} D^{\beta} f D^{\gamma} g, \qquad  \alpha \in \mathds{N}^{d}.
\end{equation}

Let $\mathfrak{h}_{n}=\mathds{R}^{2n+1}$ be the Heisenberg Lie algebra with the commutator
\begin{equation*}
[x,y]=(0,....,0,\{x,y\}), \qquad x,y \in \mathfrak{h}_{n},
\end{equation*} 
where $\{x,y\}=\sum_{i=1}^{n} (x_{i}y_{n+i}-x_{n+i}y_{i}))$, and the Heisenberg group with the Campbell-Hausdorff multiplication. In that case there is a simpler form of $f \# g$ (cf. G{\l}owacki \cite{Glo4}, Example 3.3)
\begin{equation} \label{heishas}
f \# g (w,\lambda) = c_{n} \int\int f(w+\lambda^{1 \over 2}u,\lambda) g(w+\lambda^{1 \over 2}v,\lambda) e^{i\{u,v\}} du dv, \ w \in \mathds{R}^{2n}, \lambda>0.
\end{equation}
By the chain rule and integration by parts one gets
\begin{equation} \label{tnpo}
D_{2n+1} (f \# g) =  D_{2n+1}f \# g + f \# D_{2n+1} g + {1 \over 2} \sum_{i=1}^{n} (D_{i}f \# D_{n+i}g - D_{n+i}f \# D_{i}g). 
\end{equation}
A general formula for $D^{\alpha}(f \# g)$, $\alpha \in \mathds{N}^{2n+1}$, seems to be more complicated. 

The purpose of this note is to find such a "Leibniz's formula" in the case of two-step nilpotent Lie groups, that includes the Heisenberg group. By the Fourier transform this is equivalent to find a formula for $T^{\alpha}(f*g)$, where 
$T^{\alpha}f(x)=x^{\alpha}f(x)$ and $*$ is the convolution on the group $\mathfrak{g}$.

In the Abelian case, there is a formula for the convolution product corresponding to (\ref{leib})
\begin{equation} \label{conbin}
T^{\alpha} (f*_{0}g) =\sum_{\beta+\gamma=\alpha} {\alpha \choose \beta} T^{\beta} f *_{0} T^{\gamma}g,
\end{equation}
where $*_{0}$ is the standard convolution on $\mathds{R}^{d}$.

In the general case of nilpotent Lie groups G{\l}owacki \cite{Glo2} showed that   
\begin{equation} \label{gloconv}
T^{\alpha}(f*g)=T^{\alpha}f*g+f*T^{\alpha}g+\sum_{\substack{l(\beta)+l(\gamma)=l(\alpha) \\ 0<l(\beta)<l(\alpha)}} c_{\beta,\gamma}T^{\beta}f*T^{\gamma}g,
\end{equation}
for $\alpha \neq 0$ and Schwartz functions $f$, $g$ on $\mathfrak{g}$.
Here, $c_{\beta,\gamma}$ are real constants and $l(\alpha)$ is the homogeneous length of a multiindex $\alpha$ (see Section {\ref{sectwo}}).
Notice that this formula does not give exact values of $c_{\beta,\gamma}$,
and the condition $l(\beta)+l(\gamma)=l(\alpha)$ does not characterize precisely the pairs $(\beta,\gamma)$ which appear in (\ref{gloconv}) with a nonzero constant term $c_{\beta,\gamma}$.

In order to formulate the main result we introduce some notation.
Let $X_{1},...,X_{d}$ be a base of the vector space $\mathfrak{g}$. Suppose that $A=(a_{i,j,k})_{i,j,k}$ is the matrix of the structure constants of $\mathfrak{g}$ which are given by
\begin{equation*}
[X_{i},X_{j}]=\sum_{k=1}^{d}a_{i,j,k}X_{k}, \quad\ 1 \leq i,j \leq d.
\end{equation*}
Let
$D = \{(i,j,k): a_{i,j,k} \neq 0 \}$ and $\sigma \in \mathds{N}^{D}$.
By $\sigma_{[0]}, \sigma_{[1]}, \sigma_{[2]},  \in \mathds{N}^{d}$ we denote the multiindices
\begin{equation*}
\sigma_{[0],k}= \sum_{i,j} \sigma_{(i,j,k)}, \quad\ 
\sigma_{[1],i}= \sum_{j,k} \sigma_{(i,j,k)}, \quad\
\sigma_{[2],j}= \sum_{i,k} \sigma_{(i,j,k)}.
\end{equation*} 
For $\alpha, \beta \in \mathds{N}^{d}, \sigma \in \mathds{N}^{D}$ and $\beta + \sigma_{[0]} \leq \alpha$ we define the generalized multinomial coefficient
\begin{equation} \label{gmc}
{\alpha \choose \beta}_{\sigma} = {\alpha! \over \beta! \sigma! (\alpha-\beta-\sigma_{[0]})!}.
\end{equation}
Note that in the case of the Abelian group we have $\sigma_{[0]}=\sigma_{[1]}=\sigma_{[2]}=\textbf{0}$ and
${\alpha \choose \beta}_{\sigma}={\alpha \choose \beta}$.

Our main result is the following.
\begin{thm} \label{thmaith}
Suppose that $\mathfrak{g}$ is a two-step nilpotent Lie group with the Campbell-Hausdorff multiplication. For any Schwartz functions $f$, $g$ on $\mathfrak{g}$ and every multiindex $\alpha \in \mathds{N}^{d}$,
\begin{equation} \label{thmain}
T^{\alpha}(f*g)
=\sum_{\beta+\gamma+\sigma_{[0]} = \alpha }
{\alpha \choose \beta}_{\sigma}
c_{\sigma}
T^{\beta+\sigma_{[1]}}f
*T^{\gamma +\sigma_{[2]}}g,
\end{equation}
where the (nonzero) constants $c_{\sigma}$ are given by
\begin{equation*}
c_{\sigma}= 2^{-\sum_{i,j,k} \sigma_{(i,j,k)}} \prod_{i,j,k} a_{i,j,k}^{\sigma_{(i,j,k)}},
\qquad \sigma \in \mathds{N}^{D}.
\end{equation*}
\end{thm}
An analogous formula for more than two functions is given in Proposition $\ref{encon}$ below. Moreover, in Corollary \ref{mulcons}, we show that the above formula is still valid for tempered distributions whose convolution with the Schwartz class functions is the Schwartz class.

Applying the Fourier transform to (\ref{thmain}) we get an equivalent formula for $D^{\alpha}(f \# g)$ for Schwartz functions $f, g$ on the dual $\mathfrak{g}^{*}$.
We extend this formula to the certain class of functions, that includes the classes of symbols $S^{\textbf{m}}(\mathfrak{g}^{*},\textbf{g})$ which are admissible in calculus of G{\l}owacki \cite{Glo3} (see Subsection \ref{secfifo}). 

In Subsection \ref{secfifi} we illustrate results in the case of the Heisenberg group.

\section{Two-step nilpotent Lie group} \label{sectwo}
Let $\mathfrak{g}$ be a Lie algebra of the dimension $d$ endowed with a family of one-parameter group automorphisms $(\delta_{t})_{t>0}$ which are called dilations. Let $p_{1}=1$, $p_{2}=2$ be the exponents of homogeneity of the dilations. Let
\begin{equation*}
\mathfrak{g}_{1}=\{x \in \mathfrak{g}: \delta_{t}(x)=t^{p_{1}}x\}, \qquad\ 
\mathfrak{g}_{2}=\{x \in \mathfrak{g}: \delta_{t}(x)=t^{p_{2}}x\}.
\end{equation*}
Then $\mathfrak{g}=\mathfrak{g}_{1} \oplus \mathfrak{g}_{2}$ and $\mathfrak{g}$ is a two-step nilpotent Lie algebra. Let $d_{1}=\dim \mathfrak{g}_{1}$.

The vector space $\mathfrak{g}$ is also regarded as a Lie group with the multiplication 
\begin{equation*}
x \circ y = x + y + {1 \over 2}[x,y].
\end{equation*}
The exponential map is then the identity map.
From the antisymmetry and the Jacobi identity,
\begin{equation*}
a_{i,j,k}+a_{j,i,k}=0, \qquad
\sum_{k} (a_{i,j,k}a_{k,l,m}+a_{j,l,k}a_{l,i,m}+a_{l,i,k}a_{k,j,m})=0.
\end{equation*}
Moreover, the homogeneous structure of $\mathfrak{g}$ gives that $a_{i,j,k}=0$ if any of the conditions $i=j$, $\max(i,j) \geq k$, $\max(i,j) > d_{1}$, $k \leq d_{1}$ is satisfied.
For every $k > d_{1}$ we have $(x \circ y)_{k}=x_{k}+y_{k}+r_{k}(x,y)$, where
\begin{equation*}
r_{k}(x,y) = {1 \over 2} \sum_{i=1}^{d}\sum_{j=1}^{d} a_{i,j,k} x_{i}y_{j}.
\end{equation*}

Let $T_{j}f(x)=x_{j}f(x)$, $D_{j}f(x)=i\partial_{j}f(x)$ and
\begin{equation*} 
T^{\alpha}f(x)=x_{1}^{\alpha_{1}}...x_{d}^{\alpha_{d}}f(x), \qquad
D^{\alpha}f(x)=D^{\alpha_{1}}_{1}...D^{\alpha_{d}}_{d}f(x).
\end{equation*}
Let $|\alpha|=\sum_{i=1}^{d} \alpha_{i}$ be the length of $\alpha \in \mathds{N}^{d}$. Let us also denote by $l(\alpha)$ the homogeneous length of $\alpha$, i.e.
\begin{equation*}
l(\alpha)=p_{1}(\alpha_{1}+...+\alpha_{d_{1}})+p_{2}(\alpha_{d_{1}+1}+...+\alpha_{d}).
\end{equation*}

The Schwartz space is denoted by $\mathcal{S}(\mathfrak{g})$.
Let Lebesgue measures $dx, d\xi$ on $\mathfrak{g}$ and $\mathfrak{g}^{*}$ be normalized so that the relationship between a function $f \in \mathcal{S}(\mathfrak{g})$ and its Abelian Fourier transform $\widehat{f} \in \mathcal{S}(\mathfrak{g}^{*})$ is given by
\begin{equation*}
\widehat{f}(\xi)=\int_{\mathfrak{g}} e^{-ix\xi} f(x) dx, \qquad\ 
f(x)=\int_{\mathfrak{g}^{*}} e^{ix\xi} \widehat{f}(\xi) d\xi.
\end{equation*}
The Fourier transform extends by duality to the space of tempered distributions.

A normalized Lebesgue measure on the vector space $\mathfrak{g}$ is a Haar measure on the group $\mathfrak{g}$.
The convolution $*$ on $\mathfrak{g}$ is given by
\begin{equation} \label{congr} 
f*g(x)=\int_{\mathfrak{g}} f(x\circ y^{-1})g(y) dy.
\end{equation}

Recall some notation that we have already introduced in Section {\ref{secone}}. For the group $\mathfrak{g}$ and $\sigma \in \mathds{N}^{D}$ we defined the $d$-dimensional multiindices $\sigma_{[0]}, \sigma_{[1]}, \sigma_{[2]} \in \mathds{N}^{d}$. We also defined the generalized multinomial coefficient ${\alpha \choose \beta}_{\sigma}$ for $\alpha, \beta \in \mathds{N}^{d}$ and $\sigma \in \mathds{N}^{D}$. 
Let us also denote by $c_{\sigma}$ the constants which appeared in (\ref{thmaith}), i.e.
\begin{equation} \label{consts}
c_{\sigma}= 2^{-\sum_{i,j,k} \sigma_{(i,j,k)}} \prod_{i,j,k}a_{i,j,k}^{\sigma_{(i,j,k)}}, \qquad\ \sigma \in \mathds{N}^{D}.
\end{equation}

\section{Leibniz's rule}
\subsection{Multinomial theorem.}
The following proposition is a generalization of the multinomial theorem on $\mathds{R}^{d}$. This will be crucial in the proof of Theorem \ref{thmaith}.
\begin{prop} \label{prmainpr}
For any $x,y \in \mathfrak{g}$ and every multiindex $\alpha \in \mathds{N}^{d}$,
\begin{equation} \label{prmain}
(x \circ y)^{\alpha}
=\sum_{\beta+\gamma+\sigma_{[0]} = \alpha }
{\alpha \choose \beta}_{\sigma}
c_{\sigma}
x^{\beta+\sigma_{[1]}}
y^{\gamma +\sigma_{[2]}}g,
\end{equation}
where the (nonzero) constants $c_{\sigma}$ are given by (\ref{consts}).
\end{prop}
\begin{proof}
Let $\alpha \in \mathds{N}^{d}$. We have
\begin{align} \label{xident}
(x \circ y)^{\alpha}&=\prod_{k=1}^{d}(x \circ y)_{k}^{\alpha_{k}}
=\prod_{l=1}^{d_{1}}(x_{l}+y_{l})^{\alpha_{l}}
\prod_{k=d_{1}+1}^{d}(x_{k}+y_{k}+r_{k}(x,y))^{\alpha_{k}}
\\ \nonumber
&= \prod_{l=1}^{d_{1}} \sum_{\beta_{l}+\gamma_{l} = \alpha_{l}} {\alpha_{l} \choose \beta_{l}} x_{l}^{\beta_{l}}y_{l}^{\gamma_{l}}
\prod_{k=d_{1}+1}^{d} \sum_{\substack{\beta_{k}+\gamma_{k} \\ +\tau_{k}=\alpha_{k}}} {\alpha_{k} \choose \beta_{k} \gamma_{k} \tau_{k}} x_{k}^{\beta_{k}} y_{k}^{\gamma_{k}}r_{k}(x,y)^{\tau_{k}}
\\ \nonumber
&=\sum_{\substack{{\{\beta_{l}+\gamma_{l} = \alpha_{l}:} \\ 1 \leq l \leq d_{1}\}}} \sum_{\substack{\{\beta_{k}+\gamma_{k}+\tau_{k}=\alpha_{k}: \\ d_{1}+1 \leq k \leq d\}}}
\prod_{l=1}^{d_{1}} {\alpha_{l} \choose \beta_{l}} \prod_{k=d_{1}+1}^{d} {\alpha_{k} \choose \beta_{k} \gamma_{k} \tau_{k}}
\\ \nonumber
& \times \prod_{l=1}^{d_{1}} x_{l}^{\beta_{l}} \prod_{k=d_{1}+1}^{d} x_{k}^{\beta_{k}}
\prod_{l=1}^{d_{1}} y_{l}^{\gamma_{l}} \prod_{k=d_{1}+1}^{d} y_{k}^{\gamma_{k}} 
\prod_{k=d_{1}+1}^{d} r_{k}(x,y)^{\tau_{k}}. 
\end{align}
Let $D_{k}=\{(i,j): (i,j,k) \in D\}$. Clearly, $(i,j) \in D_{k}$ if $a_{i,j,k} \neq 0$. Thus, 
\begin{align} \label{restttw}
r_{k}(x,y)^{\tau_{k}} &= \left( {1 \over 2} \sum_{(i,j)\in D_{k}} a_{i,j,k} x_{i}y_{j} \right)^{\tau_{k}}
\\ \nonumber
&=2^{-\tau_{k}} \sum_{\sum_{(i,j)\in D_{k}}\tau_{k,i,j}=\tau_{k}}  {\tau_{k} \choose ...\tau_{k,i,j}...} \prod_{(i,j) \in D_{k}} (a_{i,j,k} x_{i}y_{j})^{\tau_{k,i,j}}
\\ \nonumber
&=2^{-\tau_{k}} \sum_{\sum_{(i,j)\in D_{k}} \tau_{k,i,j}=\tau_{k}}  {\tau_{k} \choose ...\tau_{k,i,j}...}
\prod_{(i,j)\in D_{k}} a_{i,j,k}^{\tau_{k,i,j}}
\prod_{(i,j)\in D_{k}} x_{i}^{\tau_{k,i,j}}  y_{j}^{\tau_{k,i,j}}.
\end{align}
Here, ${\tau_{k} \choose ...\tau_{k,i,j}...}$ denotes a multinomial coefficient
\begin{equation*}
{\tau_{k} \choose ...\tau_{k,i,j}...} = {\tau_{k}! \over \prod_{(i,j) \in D_{k}} \tau_{k,i,j}!}.
\end{equation*}
By using (\ref{restttw}), the expression from (\ref{xident}) is equal to
\begin{align} \label{identx}
&\sum_{\substack{\{\beta_{l} + \gamma_{l} = \alpha_{l}: \\ 1 \leq l \leq d_{1}\}}} \sum_{\substack{\{\beta_{k}+\gamma_{k}+\tau_{k}=\alpha_{k}: \\ d_{1}+1 \leq k \leq d\}}}
\sum_{\sum_{(i,j)\in D_{k}}\tau_{k,i,j}=\tau_{k}}
\\ \nonumber
&\prod_{l=1}^{d_{1}} {\alpha_{l} \choose \beta_{l}}
\prod_{k=d_{1}+1}^{d} \left( {\alpha_{k} \choose \beta_{k} \gamma_{k} \tau_{k}}
{\tau_{k} \choose ...\tau_{k,i,j}...} 2^{-\tau_{k}} \prod_{(i,j) \in D_{k}}  
 a_{i,j,k}^{\tau_{k,i,j}} \right)
\\ \nonumber
\times &\prod_{l=1}^{d_{1}} x_{l}^{\beta_{l}}
\prod_{k=d_{1}+1}^{d} \left( x_{k}^{\beta_{k}}
\prod_{(i,j) \in D_{k}} x_{i}^{\tau_{k,i,j}} \right)
\prod_{l=1}^{d_{1}} y_{l}^{\gamma_{l}}
\prod_{k=d_{1}+1}^{d} \left( y_{k}^{\gamma_{k}}
\prod_{(i,j)\in D_{k}} y_{j}^{\tau_{k,i,j}} \right).
\end{align} 
If we denote $\sigma_{(i,j,k)}=\tau_{k,i,j}$, then $\sigma \in \mathds{N}^{D}$. Moreover,
\begin{equation*}
\prod_{l=1}^{d_{1}} {\alpha_{l} \choose \beta_{l}} \prod_{k=d_{1}+1}^{d} {\alpha_{k} \choose \beta_{k} \gamma_{k} \sigma_{k}} {\sigma_{k} \choose ...\sigma_{(i,j,k)}...}
= {\alpha! \over \beta! \gamma! \sigma!}
={\alpha \choose \beta}_{\sigma}.
\end{equation*}
The conditions $\beta_{l} + \gamma_{l} = \alpha_{l}$, $l=1,2,...,d_{1}$ and $\sum_{(i,j)\in D_{k}}\tau_{k,i,j}=\tau_{k}$, $\beta_{k}+\gamma_{k}+\tau_{k}=\alpha_{k}$, $k=d_{1}+1,...,d$ we can simply write as $\beta+\gamma+\sigma_{[0]} = \alpha$.
Recall that the numbers $c_{\sigma}$ are given by (\ref{consts}).
Thus, (\ref{identx}) is equal to
\begin{align*}
&\sum_{\beta + \gamma + \sigma_{[0]} = \alpha} {\alpha \choose \beta}_{\sigma} c_{\sigma}
\prod_{k=1}^{d} \left( x_{k}^{\beta_{k}}
\prod_{(i,j) \in D_{k}} x_{i}^{\sigma_{(i,j,k)}} \right)
\prod_{k=1}^{d} \left( y_{k}^{\gamma_{k}}
\prod_{(i,j) \in D_{k}} y_{j}^{\sigma_{(i,j,k)}} \right)
\\ \nonumber
=&\sum_{\beta + \gamma + \sigma_{[0]} = \alpha} {\alpha \choose \beta}_{\sigma} c_{\sigma}
\prod_{i=1}^{d} x_{i}^{\beta_{i}+\sum_{j,k:(i,j)\in D_{k}}\sigma_{(i,j,k)}}
\prod_{j=1}^{d} y_{j}^{\gamma_{j}+\sum_{i,k:(i,j)\in D_{k}}\sigma_{(i,j,k)}}
\\ \nonumber
=&\sum_{\beta+\gamma+\sigma_{[0]} = \alpha} {\alpha \choose \beta}_{\sigma} c_{\sigma}
x^{\beta+\sigma_{[1]}}
y^{\gamma+\sigma_{[2]}}.
\end{align*}
\end{proof}
 
\subsection{Convolution rule.}
\begin{proof}[Proof of Theorem \ref{thmaith}]
By (\ref{congr}) we have
\begin{equation} \label{conmt}
T^{\alpha}(f*g)(x)=x^{\alpha}(f*g)(x)=\int_{\mathfrak{g}} x^{\alpha} f(x \circ y^{-1}) g(y) dy.
\end{equation}
Applying the formula (\ref{prmain}) we get
\begin{equation} \label{iden}
x^{\alpha} = ((x \circ y^{-1}) \circ y)^{\alpha}
=\sum_{\beta+\gamma+\sigma_{[0]} = \alpha }
{\alpha \choose \beta}_{\sigma}
c_{\sigma}
(x \circ y^{-1})^{\beta+\sigma_{[1]}}
y^{\gamma +\sigma_{[2]}}.
\end{equation}
The thesis follows from combining (\ref{iden}) and (\ref{conmt}). 
\end{proof}

As a consequence, we get the relationship between exponents $\beta+\sigma_{[1]}$ and $\gamma+\sigma_{[2]}$ on the right hand side in (\ref{thmain}) in terms of homogeneous length, as in the formula (\ref{gloconv}).
\begin{cor}
The formula (\ref{gloconv}) holds.
\end{cor}
\begin{proof}
Let $\beta+\gamma+\sigma_{[0]}=\alpha$, where $\alpha,\beta,\gamma \in \mathds{N}^{d}$, $\sigma \in \mathds{N}^{D}$. By a direct calculation,
\begin{multline*}
l(\beta+\sigma_{[1]})+l(\gamma+\sigma_{[2]})
= \sum_{i=1}^{d_{1}} ( \beta_{i}+\sum_{j,k} \sigma_{(i,j,k)} ) + 2\sum_{k=d_{1}+1}^{d}\beta_{k} 
\\ \nonumber
+ \sum_{j=1}^{d_{1}} ( \gamma_{j}+\sum_{i,k}\sigma_{(i,j,k)} ) +2\sum_{k=d_{1}+1}^{d} \gamma_{k} 
=\sum_{i=k}^{d_{1}}\alpha_{k} + 2\sum_{i=d_{1}+1}^{d} \alpha_{k} = l(\alpha).
\end{multline*}
\end{proof}
If we compare the coefficients on the both sides of the formulas $T^{\alpha^{1}+\alpha^{2}}(f*g) = T^{\alpha_{1}}(T^{\alpha_{2}}(f*g))$, obtained from Theorem $\ref{thmaith}$, we get the following identity.
\begin{cor}
For any $\alpha^{1}, \alpha^{2}, \beta \in \mathds{N}^{d}$, $\sigma \in \mathds{N}^{D}$,
\begin{equation} \label{pasru}
{\alpha_{1}+\alpha_{2} \choose \beta}_{\sigma} = \sum_{b(\alpha_{1},\alpha_{2},\beta,\sigma)} {\alpha^{1} \choose \beta^{1}}_{\sigma^{1}} {\alpha^{2} \choose \beta^{2}}_{\sigma^{2}},
\end{equation}
where $b(\alpha_{1},\alpha_{2},\beta,\sigma)$ is the set
\begin{equation*}
\{(\beta^{1},\beta^{2},\sigma^{1},\sigma^{2}): \beta^{1}+\beta^{2}=\beta, \sigma^{1}+\sigma^{2}=\sigma, \beta^{1}+\sigma^{1}_{[0]} \leq \alpha^{1}, \beta^{2}+\sigma^{2}_{[0]} \leq  \alpha^{2} \}.
\end{equation*}
\end{cor}
Notice that this is the analog of the combinatorial identity
\begin{equation*}
{n_{1}+n_{2} \choose k} = \sum_{\substack{k_{1}+k_{2}=k \\ k_{1} \leq n_{1}, k_{2} \leq n_{2}}} {n_{1} \choose k_{1}} {n_{2} \choose k_{2}}, \qquad n_{1}, n_{2}, k \in \mathds{N}.
\end{equation*}

In the similar way as in Theorem \ref{thmaith} we can find a convolution rule for more than two functions.
Before that, we extend a bit our notation.
For $n \in \mathds{N}$ let
\begin{equation*}
D^{(n)}= \{(i,j,k,r,s): a_{i,j,k} \neq 0, \ \ 1 \leq r < s \leq n\}.
\end{equation*}
Notice that if $n=2$, then $D^{(2)}$ is essentially the same as $D$. 
For $\tau \in \mathds{N}^{D^{(n)}}$ we denote the multiindices in $\mathds{N}^{d}$
\begin{flalign*}
&\tau_{[0],k}= \sum_{i,j,r,s} \tau_{(i,j,k,r,s)}, \quad\ k=1,...,d
\\ \nonumber 
&\tau_{[m],l}= \sum_{j,k,s} \tau_{(l,j,k,m,s)}
+\sum_{i,k,r} \tau_{(i,l,k,r,m)}, \qquad\ m =1,...,n, l=1,...,d.
\end{flalign*} 
For $\alpha, \beta^{1},...,\beta^{n} \in \mathds{N}^{d}, \tau \in \mathds{N}^{D^{(n)}}$ and $\sum_{m=1}^{n} \beta^{m} + \tau_{[0]} = \alpha$ we denote also
\begin{equation*}
{\alpha \choose \beta^{1}...\beta^{n}}_{\tau} = {\alpha! \over \beta^{1}! ... \beta^{n}! \tau!},
\qquad
\tilde{c}_{\tau} = 2^{-\sum_{i,j,k,r,s} \tau_{(i,j,k,r,s)}} \prod_{i,j,k,r,s} a_{i,j,k}^{\tau_{(i,j,k,r,s)}}.
\end{equation*}

\begin{prop} \label{encon}
Let $f_{1},...,f_{n}$ be Schwartz functions on $\mathfrak{g}$. For every $\alpha \in \mathds{N}^{d}$,
\begin{equation}
T^{\alpha} (f_{1}*...*f_{n}) = \sum_{\beta^{1}+...+\beta^{n}+\tau_{[0]}=\alpha}
{\alpha \choose \beta^{1}...\beta^{n}}_{\tau} \tilde{c}_{\tau}
T^{\beta^{1}+\tau_{[1]}}f_{1}*...*T^{\beta^{n}+\tau_{[n]}}f_{n}.
\end{equation}
\end{prop}
\begin{proof}
In a similar fashion as in the proof of Proposition \ref{prmainpr} we find a formula for $(y^{1}\circ y^{2} \circ...\circ y^{n})^{\alpha}$, where $y^{1},...,y^{n} \in \mathfrak{g}$. We get
\begin{align} \label{ncon}
(y^{1}\circ y^{2} \circ...\circ y^{n})^{\alpha}&=\prod_{k=1}^{d}(y^{1}_{k}+...+y^{n}_{k}+{1 \over 2} \sum_{a_{i,j,k} \neq 0} a_{i,j,k} \sum_{r<s} y^{r}_{i}y^{s}_{j})^{\alpha_{k}}
\\ \nonumber
&= \sum_{\beta^{1}+...+\beta^{n}+\tau_{[0]}=\alpha}
{\alpha \choose \beta^{1}...\beta^{n}}_{\tau} \tilde{c}_{\tau}
(y^{1})^{\beta^{1}+\tau_{[1]}}...(y^{n})^{\beta^{n}+\tau_{[n]}}.
\end{align}
If we apply (\ref{ncon}) to the elements $y^{1}=x^{1}\circ(x^{2})^{-1},...,y^{n-1}=x^{n-1}\circ (x^{n})^{-1}$, $y^{n}=x^{n}$, where $x^{1},...,x^{n}$ are integral variables in the convolution, we get the thesis.
\end{proof}

\subsection{$\mathcal{S}$-convolvers.}
Let $A$ be a tempered distribution on $\mathfrak{g}$, i.e. a linear, continuous functional on $\mathcal{S}(\mathfrak{g})$. The convolution on the right by a tempered distribution $A$ with a Schwartz function $f$ on $\mathfrak{g}$ is defined by
\begin{equation*} 
f*A(x) = \langle A, \widetilde{f}_{x} \rangle, 
\end{equation*}
where $\widetilde{f}_{x}(y)=f(xy^{-1})$. $\widetilde{A}$ denotes the distribution given by $\langle \widetilde{A}, f \rangle = \langle A, \widetilde{f} \rangle$.
We say that a distribution $A \in \mathcal{S}'(\mathfrak{g})$ is a right $\mathcal{S}$-convolver on a nilpotent Lie group $\mathfrak{g}$ if $f*A \in \mathcal{S}(\mathfrak{g})$, whenever $f \in \mathcal{S}(\mathfrak{g})$. We define the space of left $\mathcal{S}$-convolvers in a similar way. $A$ is called an $\mathcal{S}$-convolver if it is both left and right $\mathcal{S}$-convolver. 
By Proposition 2.5 in Corwin \cite{Cor}, the space of $\mathcal{S}$-convolvers is closed under convolution and multiplication by polynomials. We have
\begin{equation*} 
f*(A*B)=(f*A)*B, \quad \langle A*B, f \rangle = \langle B, \widetilde{A}*f \rangle.
\end{equation*}

The formula (\ref{thmain}) is valid also for $\mathcal{S}$-convolvers instead of Schwartz functions on a two-step nilpotent Lie group.

\begin{cor} \label{cormain}
If $A$, $B$ are $\mathcal{S}$-convolvers on $\mathfrak{g}$, then,
\begin{equation} \label{mulcons}
T^{\alpha}(A*B)
=\sum_{\beta+\gamma+\sigma_{[0]} = \alpha }
{\alpha \choose \beta}_{\sigma}
c_{\sigma}
T^{\beta+\sigma_{[1]}}A
*T^{\gamma +\sigma_{[2]}}B.
\end{equation}
\end{cor}
\begin{proof}
We prove (\ref{mulcons}) by the induction on the length of $\alpha$.
Let $T^{e_{k}}=T_{k}, k=1,...,d_{1}$. Suppose at first that $A$ is a Schwartz function. Then,
\begin{equation*}
\langle T_{k}(A*B), f \rangle = \langle A*B, T_{k}f \rangle
=\langle B, \widetilde{A}*T_{k}f \rangle.
\end{equation*}
By (\ref{thmain}), it is equal to
\begin{equation*}
\langle B, T_{k}(\widetilde{A}*f) - T_{k}\widetilde{A}*f \rangle
= \langle T_{k}B, \widetilde{A}*f \rangle - \langle B, T_{k}\widetilde{A}*f \rangle.
\end{equation*}
As $\widetilde{T_{k}\widetilde{A}}=-T_{k}A$, the first step is done, when $A$ is a Schwartz function. If $A$ is an $\mathcal{S}$-convolver, then we can repeat the same reasoning using the just proven formula
\begin{equation*}
T_{k}(f*A)=T_{k}f*A+f*T_{k}A, \quad\ f \in \mathcal{S}(\mathfrak{h}),
\end{equation*}
instead of the case $\alpha=e_{k}$ in (\ref{thmain}).

Now, let $T^{e_{k}}=T_{k}$, $k=d_{1}+1,...,d$. If $A$ is a Schwartz function, then
\begin{align*}
\langle T_{k}(A*B), f \rangle &= \langle A*B, T_{k}f \rangle
=\langle B, \widetilde{A}*T_{k}f \rangle
\\
&=\langle B, T_{k}(\widetilde{A}*f) - T_{k}\widetilde{A}*f - {1 \over 2} \sum_{(i,j)\in D_{k}} a_{i,j,k} T_{i}\widetilde{A}* T_{j}f \rangle
\\
&= \langle T_{k}B, \widetilde{A}*f \rangle + \langle T_{k}A*B \rangle + {1 \over 2} \sum_{(i,j)\in D_{k}} a_{i,j,k} \langle T_{i}A*B , T_{j}f \rangle.
\end{align*}
We get $\sum_{(i,j)\in D_{k}} a_{i,j,k} T_{j}T_{i}A = 0$ from the antisymmetry of the structure constants on $\mathfrak{g}$ and then
\begin{equation} \label{mulcons1}
T_{k}(A*B)=T_{k}A*B+A*T_{k}B+{1 \over 2}\sum_{(i,j)\in D_{k}} a_{i,j,k} T_{i}A*T_{j}B,
\end{equation}
whenever $A$ is a Schwartz function. Similarly to the case $T^{e_{k}}$, for $k=1,...,d_{1}$, we obtain that (\ref{mulcons1}) also holds when $A$ is an $\mathcal{S}$-convolver.

Now, assume that the formula (\ref{mulcons}) holds for a multiindex $\alpha$. The inductive step follows from the formula $(\ref{pasru})$.
\end{proof}
 
\subsection{Leibniz's rule for the product $f \# g$.} \label{secfifo}
Applying the Fourier transform to (\ref{thmain}) we get an equivalent formula for the derivatives of the product $f \# g$ as follows.
\begin{cor} \label{derlei}
If $\alpha \in \mathds{N}^{d}$ and $f,g \in \mathcal{S}(\mathfrak{g}^{*})$, then
\begin{equation} \label{cortwi}
D^{\alpha}(f \# g)
=\sum_{\beta + \gamma + \sigma_{[0]}= \alpha }
{\alpha \choose \beta}_{\sigma}
c_{\sigma}
D^{\beta+\sigma_{[1]}}f
\# D^{\gamma +\sigma_{[2]}}g,
\end{equation}
where the constants $c_{\sigma}$ are given by (\ref{consts}).
\end{cor}

The above formula is valid under some weaker smoothness conditions for functions, what is essential for applying these results for a better understanding of the symbolic calculus on two-step nilpotent Lie groups.

Let $\textbf{m}_{1}$, $\textbf{m}_{2}$ be $\textbf{g}$-weights on $\mathfrak{g}^{*}$ (for more details see G{\l}owacki \cite{Glo3}) and 
\begin{equation*}
S^{\textbf{m}}(\mathfrak{g}^{*},\textbf{g}) = \{a \in C^{\infty}(\mathfrak{g}^{*}): |D^{\alpha} a(x)| \leq \textbf{m}(x) \rho(x)^{-l(\alpha)}\},
\end{equation*}
where $\rho(x)=1+\|x\|$, $\|\cdot\|$ being the homogeneous norm on $\mathfrak{g}$. A typical example of weight is $\textbf{m}(x)=\rho(x)^{N}$, $N \in \mathds{R}$.
Notice that if a distribution $A$ satisfies $\widehat{A} \in S^{\textbf{m}}(\mathfrak{g}^{*},\textbf{g})$ for some weight $\textbf{m}$, then one can write $A$ as a sum of a tempered distribution with compact support and a Schwartz function. Thus $A$ is an $\mathcal{S}$-convolver on $\mathfrak{g}$.
If $a \in S^{\textbf{m}_{1}}(\mathfrak{g}^{*},\textbf{g})$ and $b \in S^{\textbf{m}_{2}}(\mathfrak{g}^{*},\textbf{g})$, then, by the calculus by G{\l}owacki \cite{Glo3}, we have $a \# b \in S^{\textbf{m}_{1}\textbf{m}_{2}}(\mathfrak{g}^{*},\textbf{g})$ and a certain continuity of the product $\#$, which is sufficient to draw as a conclusion from Corollary \ref{cormain} the following.
\begin{cor}
The formula (\ref{cortwi}) holds for functions $a,b$ such that $a^{\vee}$, $b^{\vee}$ are $\mathcal{S}$-colvolvers on $\mathfrak{g}$.
In particular, if $a \in S^{\textbf{m}_{1}}(\mathfrak{g}^{*},\textbf{g})$ and $b \in S^{\textbf{m}_{2}}(\mathfrak{g}^{*},\textbf{g})$, then $D^{\alpha}(a \# b)$ is given by (\ref{cortwi}),
which also can be understood pointwise.
\end{cor}

\subsection{Heisenberg group.} \label{secfifi}
The Heisenberg group$/$algebra $\mathfrak{h}_{n}$ was introduced in Section \ref{secone}. Let us recall that the multiplication on $\mathfrak{h}_{n}$ is given by
\begin{equation} \label{multheis} 
x\circ y=(x_{1}+y_{1},...,x_{2n}+y_{2n},x_{2n+1}+y_{2n+1}+{1 \over 2} \{x,y\}).
\end{equation}

There is the remarkable relationship between the convolution structure of the Heisenberg group and the Weyl calculus for pseudodifferential operators, which was explained in, e.g., Howe \cite{How}. For $\lambda=1$ in (\ref{heishas}) one obtains the Weyl formula for the symbol of the composition of two pseudodifferential operators (cf. G{\l}owacki \cite{Glo4}, Example 3.3)
\begin{equation*}
a \#_{W} b (\xi) = \int \int a(\xi+u)b(\xi+v)e^{i\{u,v\}} du dv.
\end{equation*} 
It is easy to see that a formula for $D^{\alpha}(a \#_{W} b)$ is given by (noncommutative) Leibniz's rule
\begin{eqnarray} \label{lozen}
D^{\alpha}(a \#_{W} b) = \sum_{\beta + \gamma = \alpha} {\alpha \choose \beta} D^{\alpha}a \#_{W} D^{\gamma}b.
\end{eqnarray}\

Let $f,g \in \mathcal{S}(\mathfrak{h}_{n})$. It is directly checked that for $i=1,...,2n$
\begin{equation*} 
T_{i}(f*g)=T_{i}f*g+f*T_{i}g.
\end{equation*}
If $\alpha \in \mathds{N}^{2n+1}$ and $\alpha_{2n+1}=0$, then
\begin{eqnarray*}
T^{\alpha}(f*g)=\sum_{\beta + \gamma = \alpha} {\alpha \choose \beta} T^{\alpha}f*T^{\gamma}g,
\end{eqnarray*}
which is corresponding to (\ref{lozen}).
On the other hand, by the relation
\begin{equation*} 
x_{2n+1}=(x \circ y^{-1})_{2n+1}+ y_{2n+1} +
{1 \over 2} \sum_{i=1}^{n} ((x \circ y^{-1})_{i} y_{n+i} - (x \circ y^{-1})_{n+i} y_{i})
\end{equation*}
we get also that (cf. the formula (\ref{tnpo}))
\begin{equation*} 
T_{2n+1}(f*g)=T_{2n+1}f*g+f*T_{2n+1}g+ {1 \over 2}\sum_{i=1}^{n} (T_{i}f*T_{n+i}g-T_{n+i}f*T_{i}g).
\end{equation*}
Higher order formulas are more complicated, for instance
\begin{align*} 
T^{2}_{2n+1}(f*g)=&T^{2}_{2n+1}f*g+f*T^{2}_{2n+1}g + 2T_{2n+1}f*T_{2n+1}g
\\ +&\sum_{i=1}^{n} (T_{2n+1}T_{i}f*T_{n+i}g+T_{i}f*T_{2n+1}T_{n+i}g 
\\
-&T_{2n+1}T_{n+i}f*T_{i}g-T_{n+i}f*T_{2n+1}T_{i}g)
\\
+&{1 \over 4}\sum_{i=1}^{n}\sum_{j=1}^{n} (T_{j}T_{i}f*T_{n+j}T_{n+i}g-T_{n+j}T_{i}f*T_{j}T_{n+i}g
\\
-&T_{j}T_{n+i}f*T_{n+j}T_{i}g+T_{n+j}T_{n+i}f*T_{j}T_{i}g).
\end{align*}

We find a general formula for $T_{2n+1}^{k}(f*g)$, $k \in \mathds{N}$, as a conclusion from Theorem \ref{thmaith}. Let us first illustrate the notation by using them in the case of the Heisenberg group.
The matrix $A$ is given by
\begin{equation*}
a_{i,n+i,2n+1} = 1, \quad a_{n+i,i,2n+1} = -1, \quad i=1,...,n,
\end{equation*}
and $a_{i,j,k}=0$, otherwise. We have
\begin{equation*}
D=\{(1,n+1,2n+1),...,(n,2n,2n+1),(n+1,1,2n+1),...,(2n,n,2n+1)\}.
\end{equation*}
Let $\sigma \in \mathds{N}^{D}$. Then $\sigma_{[1]}$, $\sigma_{[2]}$, $\sigma_{[0]}$ are given by
\begin{flalign*}
&\sigma_{[1]}= (\sigma_{(1,n+1,2n+1)},...,\sigma_{(2n,n,2n+1)}, 0), \
\sigma_{[2]}= (\sigma_{(n+1,1,2n+1)},...,\sigma_{(n,2n,2n+1)}, 0),
\\
&\sigma_{[0]}= (0,...,0, \sum_{i=1}^{n}(\sigma_{(i,n+i,2n+1)}+\sigma_{(n+i,i,2n+1)})).
\end{flalign*} 
If $\sigma_{c}=\sigma_{[0],2n+1}$,
then $T_{2n+1}^{k}(f*g)$, $k \in \mathds{N}$, is given by
\begin{multline*}
T^{k}_{2n+1}(f*g)
=\sum_{\substack{\{l,m \in \mathds{N}, \sigma \in \mathds{N}^{D}: \\ l+m+\sigma_{c} = k\}}}
{k! \over l!m! \sigma!} 2^{-\sigma_{c}} (-1)^{\sum_{i=1}^{n}\sigma_{(n+i,i,2n+1)}}
\\ \nonumber
T_{1}^{\sigma_{(1,n+1,2n+1)}}...T_{2n}^{\sigma_{(2n,n,2n+1)}}T_{2n+1}^{l}f
*T_{1}^{\sigma_{(n+1,1,2n+1)}}...T_{2n}^{\sigma_{(n,2n,2n+1)}}T_{2n+1}^{m}g.
\end{multline*}

As in the procedure described in Subsection \ref{secfifo}, one gets an extension of the rule for $\mathcal{S}$-convolvers on $\mathfrak{h}_{n}$ and a formula for the derivatives of the product $a \# b$.

\subsection*{Acknowledgements.}
The author is grateful to P. G{\l}owacki and M. Preisner for their helpful remarks on the subject of the present paper. He also thanks the referee for his useful suggestions.

\end{document}